\documentclass[a4paper,11pt]{article}
\usepackage{amsmath,amsthm,amsfonts,amssymb}
\allowdisplaybreaks
\usepackage{graphicx,xcolor}
\usepackage{latexsym}
\usepackage{float}
\usepackage{setspace}
\usepackage{enumerate}
\usepackage{tikz}
\usepackage{geometry}
\usepackage{authblk}
\usepackage{hyperref}
\usepackage[capitalize]{cleveref}
\usepackage{sectsty}
\sectionfont{\centering\mdseries\normalsize}
\makeatletter
\renewcommand{\@seccntformat}[1]{\csname the#1\endcsname.\quad}
\makeatother
\usepackage{titlesec}
\titlespacing{\subsection}{0pt}{*1.5}{*0.1}
\usepackage{enumitem}
\setlist[enumerate,1]{label=(\arabic*)}

\newtheorem{theorem}{Theorem}[section]
\newtheorem{lemma}[theorem]{Lemma}
\newtheorem{definition}[theorem]{Definition}
\newtheorem{example}[theorem]{Example}
\newtheorem{corollary}[theorem]{Corollary}
\newtheorem{proposition}[theorem]{Proposition}

\newcommand{\Z}{\mathbb{Z}}

\title{\Large \textbf{On Cayley Graphs of Skew Braces and Their Distances}}

\author[1]{Ahmad Syarif Ahdal}
\author[1]{Ryan Aziz}
\author[1,2$\ast$]{Denny Riama Silaban}

\affil[1]{\small Departments of Mathematics, Faculty of Mathematics and Natural Sciences, Universitas Indonesia, Depok - Indonesia}
\affil[2]{\small Center for Research Collaboration on Graph Theory and Combinatorics, Indonesia}

\date{}

\begin{document}

\maketitle

\thispagestyle{plain}
\let\thefootnote\relax\footnotetext{
    \noindent $^*$Corresponding author: \texttt{denny@sci.ui.ac.id}\\[0.3em]
}

\begin{quote}
\noindent\textbf{Abstract.} We introduce the left and right Cayley graphs of a finite skew brace $(A,\cdot, \circ)$ with a connection set $T\subset A$, as a directed multigraph with additive and multiplicative edges, depicting the operations $\cdot$ and $\circ$ respectively. We compute the distance from $a$ to $a\circ b$ in these Cayley graphs in terms of factorization lengths over $T$. As an application, we provide a characterization of the ideals of a skew brace in terms of their corresponding Cayley graphs.

\vspace{0.5em}
\end{quote}

\vspace{0.5cm}

\section{INTRODUCTION}

Originally formulated in the study of statistical mechanics, The Yang-Baxter equation (YBE) has since become a central object in quantum groups, knot theory, and related areas. In \cite{drinfeld1992unsolved}, Drinfeld suggested studying a set-theoretical solution of the Yang-Baxter equation, which is a map $r: X\times X \to X\times X$ on a set $X$ satisfying 
$$(r\times \mathrm{id})(\mathrm{id} \times r)(r\times \mathrm{id}) = (\mathrm{id} \times r)(r\times \mathrm{id})(\mathrm{id} \times r).$$

A family of involutive non-degenerate solutions of the YBE is constructed by Rump \cite{rump2007braces} by introducing \textit{braces}, which later is generalized into \textit{skew left braces} (or simply \textit{skew braces}) by Guarniere and Vendramin \cite{guarnieri2017skew} to cover the non-involutive cases. Here a skew brace is a set $A$ equipped with two group operations $\cdot$ and $\circ$ such that $(A,\cdot)$ and $(A,\circ)$ form a group (respectively called an additive and multiplicative group), such that
\[a\circ (bc) = (a\circ b)a^{-1}(a\circ c).\]

Skew braces have found themselves in many areas, including representation theory \cite{kozakai2024representation}, Hopf-Galois theory \cite{stefanello2023connection}, homology \cite{gran2025hopf}, cohomology \cite{rathee2024cohomology}, and recently, in Lie theory \cite{damele2026simple}. Despite these advances, graphs associated with skew braces have received little attention, although there are some studies. One example is the common divisor graph for skew braces as introduced  by Properzi and Van Antwerpen \cite{properzi2025common}, whose vertices are the non-trivial orbits of the multiplicative group $(A,\circ)$, and two distinct vertices are adjacent if and only if their orbit sizes are not coprime, and they investigated the graph's diameter in connection with the left nilpotency of the skew brace. 

The graphs constructed by Properzi and Van Antwerpen, by their construction, do not visualize the interaction between two operations on individual elements and translate algebraic relations into combinatorial structures. To do this, at least for the group $G$, it is useful to use \textit{Cayley graphs}, where the vertices represent elements of the group, and two elements $u$ and $v$ are adjacent if and only if $u \cdot v^{-1} \in T$, where $T\subseteq G$ is a connection set. Extensions of Cayley graphs have also been established on other algebraic systems, such as finite rings \cite{akhtar2009unitary, madhumitha2023graphs} and dihedral groups \cite{miklavic2007distance}. Since skew braces are closely related to groups, it is natural to study Cayley graphs constructed from skew braces and their applications. 

A related study was conducted by Ballester-Bolinches et al.~\cite{ballester2021}, in which they investigated the left brace structure of the permutation group $\mathcal{G}(X,r) = \langle f_x \mid x \in X \rangle$ associated with an involutive, non-degenerate set-theoretic solution $(X,r)$ of the Yang-Baxter equation via Cayley graphs. Starting with the Cayley graph of the multiplicative group $(\mathcal{G}(X,r), \cdot)$ with respect to its natural generating set $S = \{f_x \mid x \in X\}$, they applied an edge-relabeling procedure, where each directed edge $\alpha \to \alpha f_x$, initially labeled by $x$, is relabeled by $\alpha(x)$. They proved that the resulting labeled graph forms the Cayley graph of an abelian group $(\mathcal{G}(X,r), +)$ with respect to the same generating set, where the induced addition $\alpha + f_{\alpha(x)} = \alpha f_x$ represents the left brace relation $\alpha f_x = \alpha + \lambda_\alpha(f_x)$ (analogous to $a \circ b = a \cdot \lambda_a(b)$ in a skew brace). 

In their construction, the multiplicative and additive Cayley graphs share the exact same edge sets and differ only in their edge labels. Because a multiplicative directed edge $(\alpha, \alpha \circ f_x)$ corresponds to an additive directed edge by $\lambda_\alpha(f_x) = f_{\alpha(x)}$, their relabeling yields a valid Cayley graph for the additive group only when the connection set is invariant under the $\lambda$-action—a condition automatically satisfied by the natural generating set $S$. Simply relaxing their connection set to an arbitrary subset $T$ would cause the relabeled additive edges to vary across vertices as $\lambda_\alpha(T)$, failing to produce a well-defined Cayley graph with respect to $T$. Consequently, the restriction of $T$ to be a generating set makes it harder to observe internal substructures such as ideals, which in many cases are proper subsets. 

To tackle this limitation, in this paper we consider both group operations of a finite skew brace $A$ at the same time, and allow the connection set $T$ to be an arbitrary subset of $A$. We introduce the concept of Cayley graph of a finite skew brace, denoted by $\mathrm{Cay}((A, \cdot, \circ), T)$, where we define the vertex set as $A$ and the edge set as a disjoint union $E = E_\cdot \sqcup E_\circ$ corresponding to the first and second operations, respectively. Here, the adjacency of any two vertices in the graph follows the definition from Aikawa \cite{aikawa2026hashing}. We establish the characterization of paths and distances in the Cayley graph $\mathrm{Cay}((A, \cdot, \circ), T)$. Specifically, we show that the distance between vertices $a$ and $a\circ b$ corresponds to the minimal factorization length of the corresponding element $\lambda_a(b)$, where $\lambda : (A,\circ)\to \operatorname{Aut}(A,\cdot)$ is a group homomorphism (see Theorem \ref{theo:distance-cayley}).  Finally, in Theorem \ref{theorem_ideal}, we characterize the ideals of a skew brace by looking at the properties of their corresponding Cayley graphs.
\vspace{1.5 cc}
\section{PRELIMINARIES} \label{sec:preliminaries}

\subsection{Skew Braces} \label{sec:preliminaries} \mbox{} \par
Throughout this paper, by a skew brace we mean a skew left brace. There is the right version of skew brace, which can be treated analogously. We also assume that all skew braces are finite. We refer to \cite{guarnieri2017skew} for more details.
\vspace{.5 cc}
\begin{definition} \label{def:skew_brace}
\emph{A skew (left) brace} $(A,\cdot,\circ)$ consists of groups $(A,\cdot)$ and $(A,\circ)$ such that
\begin{equation}\label{eq:brace_distributive}
    a \circ (b\cdot c) = (a \circ b)\cdot  a^{-1} \cdot  (a \circ c)
\end{equation}
holds for all $a, b, c \in A$, where $a^{-1}$ denotes the inverse of $a$ in $(A,\cdot)$\end{definition}

The groups $(A, \cdot)$ and $(A,\circ)$ are traditionally called the \emph{additive group} and \emph{multiplicative group} respectively, although in our setting, the group $(A,\cdot)$ is in general not abelian. In the case where $(A, \cdot)$ is abelian, the skew brace $(A,\cdot, \circ)$ is a brace in the sense of Rump \cite{rump2007braces}.

\vspace{.5cc}

In any skew brace $(A,\cdot,\circ)$, the groups $(A,\cdot)$ and $(A,\circ)$ have the same identity element, denoted by $1$. We also denote $a^{-1}$ (resp. $a^\dagger$) as the inverse of $a$ in $(A,\cdot)$ (resp. in $(A,\circ)$). 

Any group $(A, \cdot)$ can be made into a skew brace in two ways, either into a \textit{trivial skew brace} $(A,\cdot, \cdot)$, or into an \textit{almost trivial skew brace} $(A,\cdot,\cdot^{\operatorname{op}})$, where $\cdot^{\operatorname{op}}$ denotes the opposite multiplication, i.e. $a \cdot^{\operatorname{op}} b = ba$ for any $a,b\in A$. One can easily verify that the skew brace condition holds. An almost trivial skew brace is a trivial skew brace if $(A,\cdot)$ is an abelian group.

For any skew brace, there is a group homomorphism
\begin{align}\label{lambda-map}
\lambda: (A,\circ)\to \mathrm{Aut}(A,\cdot), \qquad a \mapsto \lambda_a(b) = a^{-1}(a\circ b),
\end{align}
and using the map $\lambda$, one can show that $a \circ b = a \cdot \lambda_a(b)$, and $a \cdot b = a \circ \lambda_a^{-1}(b)$.

Here are some examples of non-trivial skew braces.
\vspace{.5cc}

\begin{example} \label{ex:z4_brace} Let $(\Z_4, +)$ be the abelian group of integers modulo 4, and define the operation $\circ$ by
\begin{equation*}
    a \circ b = a + b + 2ab \pmod 4.
\end{equation*}
One can check that $(\Z_4, +, \circ)$ is a skew brace, with $a^\dagger = -a(1 + 2a)^{-1} \equiv a(2a - 1) \pmod 4$. 
\end{example}
\vspace{.5cc}
\begin{example} [c.f. \cite{kozakai2024representation}] \label{ex:s3_brace}
Let $(S_3, \cdot)$ be the permutation group of degree $3$. We can decompose $S_3$ into the product of two subgroups $H = \langle (123) \rangle$ and $K = \langle (12) \rangle$, such that every element $a \in S_3$ factorizes uniquely as $a = a_1 \cdot a_2$ with $a_1 \in H$ and $a_2 \in K$. 
The operation $\cdot$ is a standard composition of permutations and $\circ$ is defined as $a \circ b = a_1 \cdot b \cdot a_2$, giving a non-trivial skew brace where the additive structure is non-abelian.
\end{example}

Finally, we recall the notion of ideal in skew braces as follows.
\vspace{.5 cc}
\begin{definition}\label{def:ideal_skewbrace}
    Let $A$ be a skew brace. A normal subgroup $I$ of $(A,\circ)$ is an ideal of $A$ if $I\cdot a=a \cdot I$ and $\lambda_a(I)\subseteq I$ for all $a\in A$.
\end{definition}
\vspace{.5cc}
\begin{lemma}\label{lem:propideal_skewbrace}
    Let $A$ be a skew brace and let $I\subseteq A$ to be an ideal. Then the following properties hold:
\begin{enumerate}
    \item $I$ is a normal subgroup of $(A,\cdot)$.
    \item $a\circ I=a\cdot I$, for all $a\in A$.
\end{enumerate}
\end{lemma}

\vspace{0.5cc}

\subsection{Cayley Graphs of Groups}\mbox{}\par

We assume our readers are familiar with graph theory, including Cayley graphs, and thus this section is intended to recall and fix some terminologies. We refer to  \cite{diestel2017graph, liu2022eigenvalues} for basics in graph theory.

Recall that a \emph{finite directed graph} (or \emph{digraph}) $G$ is a graph that consists of a non-empty finite vertex set $V(G)$, which contains vertices of $G$ and an edge set $E(G)$, which contains edges of $G$ in the form of ordered pairs of vertices. For an edge $(u, v) \in E(G)$, $u$ is said to be \emph{adjacent to} $v$, and $v$ is \emph{adjacent from} $u$ and the edge $(u,v)$ is said to be \emph{incident} to both $u$ and $v$. A self-incident edge of the form $(u,u)$ is referred to as a \emph{loop}, whereas multiple edges connecting the exact same pair of distinct vertices are known as \emph{parallel edges}. A graph that allows loops and parallel edges is termed a \emph{directed multigraph}, and it is called a \emph{simple directed graph} when it contains neither loops nor parallel edges. Let $G = (V, E)$ and $G' = (V', E')$ be two graphs. $G$ and $G'$ are isomorphic, and write $G \simeq G'$ if there exists a bijection $\varphi: V \to V'$ which preserves the adjacency of vertices, that is,  $(x,y) \in E \Leftrightarrow \varphi(x)\varphi(y) \in E'$.

It is known, for a given group, one can construct its Cayley graph. Following Aikawa \cite{aikawa2026hashing}, we recall the definition of left and right Cayley graphs.
\begin{definition}[c.f. \cite{aikawa2026hashing}]
Let $A$ be a finite group and let $T$ be a subset of $A$ (called a \textit{connection set}). The \emph{left Cayley graph}, denoted by $\operatorname{Cay}_L(A, T)$, is the graph with vertex set $V(\operatorname{Cay}_L(A, T)) = A$ and edge set
\begin{equation}
    E(\operatorname{Cay}_L(A, T)) = \left\{ (tv, v) \mid v \in A, \, t \in T \right\}.
\end{equation}
In contrast, the \emph{right Cayley graph}, denoted by $\operatorname{Cay}_R(A, T)$, is the graph with vertex set $A$ and edge set
\begin{equation}
    E(\operatorname{Cay}_R(A, T)) = \left\{ (vt, v) \mid v \in A, \, t \in T \right\}.
\end{equation}
\end{definition}

We note that the definition of Cayley graphs in \cite{aikawa2026hashing} uses edges $(u,tu)$ and $(u,ut)$ for the left and right Cayley graphs, whereas in our definition, we use $(tv,v)$ and $(vt,v)$. These two conventions differ only by reversing the direction of every edge. Moreover, \cite{aikawa2026hashing} requires $T \subseteq A \setminus \{1\}$,
while we allow any subset $T \subseteq A$ unless stated otherwise.

When the connection set $T$ is inverse-closed (i.e., $T = T^{-1}$), then the Cayley graph is undirected. If $1$ is also not in $T$, then the Cayley graph has no loops and hence can be identified as a simple undiricted graph. Furthermore, it is known that the Cayley graph of a group $A$ is connected if and only if the connection set $T$ generates $A$ (See Liu and Zhou \cite{liu2022eigenvalues} for more details.)
\vspace{1.5 cc}
\section{MAIN RESULTS}

In this section, we explore the definitions, properties, and examples of Cayley graphs of finite skew braces, focusing on their connection to the skew brace. 
\vspace{.5cc}
\begin{definition}\label{def:cayley-skew-brace}
Let $(A,\cdot,\circ)$ be a finite skew brace and $T \subseteq A$ a connection set.
The \emph{left Cayley graph} $\operatorname{Cay}_L((A,\cdot,\circ),T)$ is the directed multigraph
with vertex set $A$ and edge set $E_{\cdot,L} \sqcup E_{\circ,L}$, where
\begin{align*}
E_{\cdot,L} &= \{(u,v) \in A \times A \mid u \cdot v^{-1} \in T\}, &
E_{\circ,L} &= \{(u,v) \in A \times A \mid u \circ v^{\dagger} \in T\}.
\end{align*}
The \emph{right Cayley graph} $\operatorname{Cay}_R((A,\cdot,\circ),T)$ is defined in the same way,
with $E_{\cdot,L}$ and $E_{\circ,L}$ replaced by
\begin{align*}
E_{\cdot,R} &= \{(u,v) \in A \times A \mid v^{-1} \cdot u \in T\}, &
E_{\circ,R} &= \{(u,v) \in A \times A \mid v^{\dagger} \circ u \in T\}.
\end{align*}
\end{definition}

We call an element of $E_{\cdot}$ (resp. $E_{\circ}$) as an additive edge (resp. a multiplicative edge). 

One can observe that the left Cayley graph of trivial skew braces $(A,\cdot, \cdot)$ with the connection set $T\subseteq A$ is identical to two copies of $\operatorname{Cay}_L((A,\cdot), T)$ sharing the same vertex set. To see this, note that in the trivial skew brace, $v^{\dagger} = v^{-1}$ for any $v\in A$, implying that $E_{\cdot,L} = E_{\circ, L}$. Therefore, by Definition \ref{def:cayley-skew-brace}, $E= E_{\cdot, L}\sqcup E_{\circ, L} = E_{\cdot, L}\sqcup E_{\cdot,L} $, and thus every edge $(u,v)$ in $\operatorname{Cay}_L((A,\cdot),T)$ appears exactly twice in $\operatorname{Cay}_L((A,\cdot,\cdot),T)$. Moreover, if $(A,\cdot)$ is abelian, then $\operatorname{Cay}_L((A,\cdot,\cdot),T) = \operatorname{Cay}_R((A,\cdot,\cdot),T)$. The case for the right Cayley graph of trivial skew brace is similar. 

In the case of an almost trivial skew brace, we have the following proposition:
\vspace{0.5 cc}
\begin{proposition}
Let $(A, \cdot, \circ)$ be an almost trivial  skew brace where $\circ = \cdot^{\operatorname{op}}$, and let $T \subseteq A$ be a connection set. Then the following hold:
\begin{enumerate}
    \item $E_{\cdot,L}$ induces $\mathrm{Cay}_L((A, \cdot), T)$ and $\mathrm{Cay}_R((A, \circ), T)$.
    \item $E_{\cdot,R}$ induces $\mathrm{Cay}_R((A, \cdot), T)$ and $\mathrm{Cay}_L((A, \circ), T)$.
    \item $E_{\circ,L}$ induces $\mathrm{Cay}_L((A, \circ), T)$ and $\mathrm{Cay}_R((A, \cdot), T)$.
    \item $E_{\circ,R}$ induces $\mathrm{Cay}_R((A, \circ), T)$ and $\mathrm{Cay}_L((A, \cdot), T)$.
\end{enumerate}
\end{proposition}

\begin{proof}
It is obvious that $E_{\cdot,L}$ induces $\mathrm{Cay}_L((A, \cdot), T)$ by Definition \ref{def:cayley-skew-brace}. Furthermore, we have $(u,v)\in E_{\cdot, L}$ iff $uv^{-1} = v^{\dagger}\circ u \in T$ iff $(u,v)\in E_{\circ, R}$. This proved part (1). The proof for other parts is similar.
\end{proof}

In the following examples, we give some examples of Cayley graphs of non trivial finite skew braces.

\vspace{.5 cc}

\begin{example} \label{ex:cay skew brace Z4}
Consider the skew brace $(\mathbb{Z}_4,+,\circ)$ in Example \ref{ex:z4_brace} with $a\circ b=a + b + 2ab \pmod 4$. Let $T = \{3\}$. By Definition~\ref{def:cayley-skew-brace}, $(u,v)\in E_{+,L}$ iff $u-v\in \{3\}$ iff $u=3+v$. Similarly, we have $(u,v)\in E_{\circ, L}$ iff $u\circ v^{\dag} \in \{3\}$ iff $u=3\circ v$. The resulting Cayley graph is shown in Figure~\ref{fig:cayley-z4-t3}.  One can see that the additive edges set $E_{\cdot,L}$ induces a directed 4-cycle, while the multiplicative edges set $E_{\circ,L}$ induces two disjoint 2-cycle with double edges. Since $\mathbb{Z}_4$ is abelian with respect to both $+$ and $\circ$, then both $-v+u$ and $v^{\dagger}\circ u$ are in $T$, and thus $\mathrm{Cay}_L((\mathbb{Z}_4, +, \circ), T)= \mathrm{Cay}_R((\mathbb{Z}_4, +, \circ), T)$.

\begin{figure}[H]
    \centering
    \begin{tikzpicture}[
        scale=0.9, transform shape, 
        every node/.style={circle, draw=black, fill=gray!10, font=\footnotesize, minimum size=6mm, inner sep=0pt},
        >=stealth
    ]
       
        \node (0) at (-1.3, 1.3) {$0$};
        \node (1) at (1.3, 1.3) {$1$};
        \node (2) at (1.3, -1.3) {$2$};
        \node (3) at (-1.3, -1.3) {$3$};

        \draw[->, semithick, red, bend right=20] (3) to (0);
        \draw[->, semithick, blue, bend left=35] (3) to (0);
        \draw[->, semithick, blue, bend right=20] (0) to (3);
        \draw[->, semithick, red] (0) to (1);
        \draw[->, semithick, red, bend right=20] (1) to (2);
        \draw[->, semithick, blue, bend left=35] (1) to (2);
        \draw[->, semithick, blue, bend right=20] (2) to (1);
        \draw[->, semithick, red] (2) to (3);

    \end{tikzpicture}
    
    \vspace{0.15cm}
    \begin{tikzpicture}[scale=0.85, transform shape]
        \draw[->, semithick, red] (0,0) -- (0.5,0) node[right, text=black, font=\scriptsize, draw=none, fill=none] {Additive edge};
        \draw[->, semithick, blue] (3.4,0) -- (3.9,0) node[right, text=black, font=\scriptsize, draw=none, fill=none] {Multiplicative edge};
    \end{tikzpicture}

    \caption{The Cayley graph $\mathrm{Cay}_L((\mathbb{Z}_4, \cdot, \circ), \{3\})$.}
    \label{fig:cayley-z4-t3}
\end{figure}
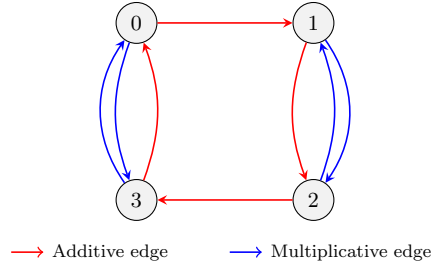
\end{example}

\vspace{0.05 mm}

\begin{example} \label{ex:cay skew brace S3}
The left and right Cayley graphs of a skew brace $(S_3,\cdot,\circ)$ (as discussed in Example \ref{ex:s3_brace}) with connection set $T = \{(13)\}$ are given in Figure~\ref{fig:cayley-s3-t13} below. For both graphs, the additive edge sets $E_{\cdot}$ induce three disjoint 2-cycles with double edges, while the multiplicative edge sets $E_{\circ}$ induce a directed 6-cycle. Since $\circ$ is not commutative, the left and right Cayley graph of $(S_3, \cdot, \circ)$ are not isomorphic.

\begin{figure}[H]
    \vspace{-0.3cm}
    \centering
    \begin{tikzpicture}[
        scale=0.9, transform shape, 
        vertex/.style={
            circle, 
            draw=black, 
            fill=gray!10, 
            font=\tiny, 
            text width=5.6mm, 
            align=center, 
            inner sep=0pt
        },
        >=stealth
    ]
    
        \begin{scope}[xshift=-3.2cm] 
            \node[font=\normalsize] at (0, 2.2) {$\mathrm{Cay}_L((S_3, \cdot, \circ), T)$};
            
            \node[vertex] (id)  at (90:1.45cm)  {$1$};
            \node[vertex] (23)  at (30:1.45cm)  {$(23)$};
            \node[vertex] (123) at (330:1.45cm) {$(123)$};
            \node[vertex] (12)  at (270:1.45cm) {$(12)$};
            \node[vertex] (132) at (210:1.45cm) {$(132)$};
            \node[vertex] (13)  at (150:1.45cm) {$(13)$};

            \draw[->, semithick, red, bend left=15] (id) to (13);
            \draw[->, semithick, red, bend left=15] (13) to (id);

            \draw[->, semithick, red, bend left=18] (12) to (123);
            \draw[->, semithick, red, bend left=18] (123) to (12);

            \draw[->, semithick, red, bend left=15] (23) to (132);
            \draw[->, semithick, red, bend left=15] (132) to (23);

            \draw[->, semithick, blue, bend left=15] (id) to (23);
            \draw[->, semithick, blue, bend left=15] (23) to (123);
            \draw[->, semithick, blue, bend left=42] (123) to (12);
            \draw[->, semithick, blue, bend left=15] (12) to (132);
            \draw[->, semithick, blue, bend left=15] (132) to (13);
            \draw[->, semithick, blue, bend left=45] (13) to (id);
        \end{scope}

        \begin{scope}[xshift=3.2cm]  
            \node[font=\normalsize] at (0, 2.2) {$\mathrm{Cay}_R((S_3, \cdot, \circ), T)$};
            
            \node[vertex] (id_R)  at (90:1.45cm)  {$1$};
            \node[vertex] (23_R)  at (30:1.45cm)  {$(23)$};
            \node[vertex] (123_R) at (330:1.45cm) {$(123)$};
            \node[vertex] (12_R)  at (270:1.45cm) {$(12)$};
            \node[vertex] (132_R) at (210:1.45cm) {$(132)$};
            \node[vertex] (13_R)  at (150:1.45cm) {$(13)$};

            \draw[->, semithick, red, bend left=18] (id_R) to (13_R);
            \draw[->, semithick, red, bend left=18] (13_R) to (id_R);

            \draw[->, semithick, red, bend left=18] (132_R) to (12_R);
            \draw[->, semithick, red, bend left=18] (12_R) to (132_R);

            \draw[->, semithick, red, bend left=18] (123_R) to (23_R);
            \draw[->, semithick, red, bend left=18] (23_R) to (123_R);

            \draw[->, semithick, blue, bend left=42] (id_R) to (13_R);
            \draw[->, semithick, blue, bend left=15] (13_R) to (132_R);
            \draw[->, semithick, blue, bend left=42] (132_R) to (12_R);
            \draw[->, semithick, blue, bend left=15] (12_R) to (123_R);
            \draw[->, semithick, blue, bend left=42] (123_R) to (23_R);
            \draw[->, semithick, blue, bend left=15] (23_R) to (id_R);
        \end{scope}

    \end{tikzpicture}
    
    \vspace{0.15cm}
    \begin{tikzpicture}[scale=0.85, transform shape]
        \draw[->, semithick, red] (0,0) -- (0.5,0) node[right, text=black, font=\scriptsize] {Additive edge};
        \draw[->, semithick, blue] (3.4,0) -- (3.9,0) node[right, text=black, font=\scriptsize] {Multiplicative edge};
    \end{tikzpicture}

    \vspace{-0.15cm}
    \caption{The left and right Cayley graphs of the skew brace $S_3$ with connection set $T=\{(13)\}$.}
    \label{fig:cayley-s3-t13}
    \vspace{-0.3cm}
\end{figure}
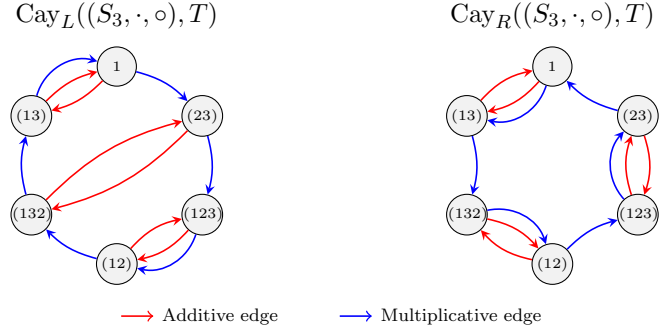
\end{example}

\vspace{0.5 cc}

\subsection{Path and Distance Characterizations}\mbox{}\par
We now define paths and distances in Cayley graphs of skew braces. In this section, a path refers to a directed path, following the terminology in \cite{diestel2017graph}.
\vspace{0.5 cc}
\begin{definition}\label{def:paths_distances}
Let $A=(A,\cdot,\circ)$ be a finite skew brace, $T\subseteq A$, and $a,b\in A$. 
\begin{itemize}
    \item An \emph{additive} (resp.\ \emph{multiplicative}) \emph{path} of length $k \ge 0$ from $a$ to $b$ in $\operatorname{Cay}_L((A,\cdot,\circ),T)$ is a sequence of distinct vertices $a=u_0, u_1, \dots, u_k=b$ such that $(u_{i-1}, u_i) \in E_{\cdot,L}$ (resp.\ $E_{\circ,L}$) for all $i \in \{1, \dots, k\}$.
    
    \item The \emph{additive distance} $d_{\cdot,L}(a,b)$ (resp.\ \emph{multiplicative distance} $d_{\circ,L}(a,b)$) from $a$ to $b$ is the length of a shortest additive (resp.\ multiplicative) path from $a$ to $b$. If no such path exists, then $d_{\cdot,L}(a,b) = \infty$ (resp.\ $d_{\circ,L}(a,b) = \infty$).
\end{itemize}
Paths and distances $d_{\cdot,R}(a,b)$ and $d_{\circ,R}(a,b)$ in $\operatorname{Cay}_R(A,T)$ are defined analogously.
\end{definition}

We also define a length of an element $a \in A$ as follows.
\vspace{0.5 cc}
\begin{definition}\label{def:element_length}
Let $A=(A,\cdot,\circ)$ be a finite skew brace and $T\subseteq A$. For $a\in A$, the \emph{additive length} $\ell_{\cdot,T}(a)$ (resp.\ \emph{multiplicative length} $\ell_{\circ,T}(a)$) is the smallest $k\geq 0$ such that
\[
a = t_1 \cdot t_2 \cdot ... \cdot t_k \quad (\text{resp. } a = t_1 \circ t_2\cdot ... \cdot  t_k),
\]
for some $t_1, \dots, t_k \in T$, such that $\ell_{\cdot, T}(1) = \ell_{\circ, T}(1) = 0$,
and $\ell_{\cdot,T}(a)=\infty$ (resp. $\ell_{\circ,T}(a)=\infty$) if no such $k$ exists.
\end{definition}
Clearly both the additive and multiplicative length of $a\ne 1\in T$ is $1$.

We have the following result:
\begin{theorem}\label{theo:distance-cayley}
For all $a, b\in A$ we have:
\begin{align}
  d_{\cdot, L}(a, a \circ b) =& \ell_{\cdot, T}(a \cdot \lambda_a(b)^{-1} \cdot a^{-1}),\label{eq:dL,cdot}\\
  d_{\cdot, R}(a, a \circ b) =& \ell_{\cdot, T}(\lambda_a(b)^{-1}), \label{eq:dR,cdot}\\
  d_{\circ, L}(a,a\circ b)=& \ell_{\circ, T}(a\circ b^{\dagger}\circ a^{\dagger}),\label{eq:dL,circ}\\
d_{\circ, R}(a, a\circ b)=& \ell_{\circ, T}(b^{\dagger}).\label{eq:dR,circ}
\end{align}
\end{theorem}

\begin{proof} Suppose $d_{\cdot, L}(a,a\circ b)=k$. It means that there are intermediate vertices $v_1,v_2,..., v_{k-1}$ and generators $t_1, t_2,..., t_k \in T$ such that $t_1 = a\cdot v_1^{-1}$, $t_k = v_{k-1}\cdot (a\circ b)^{-1}$, and $t_j = v_j \cdot v_{j+1}^{-1}$ for $2\leq j\leq k-1$.
Through backward substitutions, we have
$$a = (t_1 \cdot t_2 \cdot ... \cdot t_k) \cdot (a \circ b).$$
Using the fact that $a \circ b = a \cdot \lambda_a(b)$, we obtain
$$   a \cdot (a \circ b)^{-1} = a \cdot (a \cdot \lambda_a(b))^{-1} = a \cdot (\lambda_a(b)^{-1} \cdot a^{-1}) = a \cdot \lambda_a(b)^{-1} \cdot a^{-1}.$$
Therefore we have $a \cdot \lambda_a(b)^{-1} \cdot a^{-1} = t_1 \cdot t_2 \cdot ... \cdot t_k$, which corresponds to the shortest factorization with length $k$ under the additive operation $\cdot$. We conclude that $\ell_{\cdot, T}(a \cdot \lambda_a(b)^{-1} \cdot a^{-1})=k$ and proved Equation \ref{eq:dL,cdot}.

The proof for Equation \ref{eq:dR,cdot}, \ref{eq:dL,circ}, and \ref{eq:dR,circ} are similar, by using the appropriate left or right Cayley graph under the operation $\cdot$ or $\circ$. 
\end{proof}

It is known that if a connection set $T$ is closed under conjugation in $(A, \cdot)$, then $\mathrm{Cay}_R((A, \cdot), T)\simeq \mathrm{Cay}_L((A, \cdot), T) $ \cite{fostergreenwood2015spectra}. Consequently, the minimum factorization length of any element over $T$ is identical to that of its conjugate. Applying this to our case, the distance from $a$ to $a \circ b$ in the $\mathrm{Cay}_R((A, \cdot), T)$ is identical to its corresponding distance in the $\mathrm{Cay}_L((A, \cdot), T)$.

Furthermore, if $T$ is self-inverse with respect to additive (resp. multiplicative) operation, meaning $T=T^{-1}$ (resp. $T=T^\dagger$) we have the following corollary:
\vspace{.5cc}
\begin{corollary}\label{cor:lambda_as_path}
    If $T=T^{-1}$ then $d_{\cdot, R}(a,a\circ b)=\ell_{\cdot, T}(\lambda_a(b))$. Similarly, if $T=T^\dagger$ then $d_{\circ, R}(a,a\circ b)=\ell_{\circ, T}(b)$.
\end{corollary}

\begin{proof}
From Theorem~\ref{theo:distance-cayley} Equation \ref{eq:dR,cdot}, we have $\lambda_a(b)^{-1} = t_k \cdot t_{k-1} \cdot ... \cdot t_1$ for some $t_1, ... , t_k \in T$, hence
\[\lambda_a(b) = t_1^{-1} \cdot t_2^{-1} \cdot ... \cdot t_k^{-1}.\]
Since $T=T^{-1}$, there exists $s_i=t_i^{-1}\in T$ such that
$\lambda_a(b) = s_1 \cdot s_2 \cdot ... \cdot s_k$, and thus we obtain the distance from $a$ to $a\circ b$ in $\operatorname{Cay}_R((A,\cdot),T)$ as stated. The proof for the other part is similar, using $T=T^{\dagger}$.
\end{proof} 

By the first part of Corollary~\ref{cor:lambda_as_path}, one can see $\lambda_a(b)$ in skew braces as a path in the Cayley graph of skew brace $\mathrm{Cay}_R((A, \cdot, \circ), T)$.

We then show that the adjacency of $a$ and $b$ via additive (resp. multiplicative) edge in left and right Cayley graph of finite skew brace induces some multiplicative (resp. additive) paths. 
\vspace{.5 cc}
\begin{theorem} \label{theo:mixed-edge-additive/multiplicative-path-ab-left-cayley}
Let $\mathrm{Cay}_L((A, \cdot, \circ), T)$ be a left Cayley graph of a finite skew brace $(A, \cdot, \circ)$ with connection set $T \subseteq A$. For any $a, b \in A$, the following hold: 
\begin{enumerate}
    \item If $\langle T\rangle_\cdot=(A,\cdot)$ and $d_{\circ,L}(a,b)=1$, then $a=t\circ b$ for some $t\in T$, such that $d_{\cdot,L}(a,b) = \ell_{\cdot,T}(t \cdot \lambda_t(b) \cdot b^{-1})$ and $d_{\cdot,L}(b,a) = \ell_{\cdot,T}(b \cdot \lambda_t(b)^{-1} \cdot t^{-1})$. Consequently, if $\ell_{\cdot,T}(b \cdot \lambda_t(b)^{-1} \cdot t^{-1}) = k$, then the graph contains a directed $(k+1)$-cycle consisting of $k$ additive edges and one multiplicative edge.
    \item If $\langle T\rangle_\circ=(A,\circ)$ and $d_{\cdot,L}(a,b)=1$, then $a=t\cdot b$ for some $t\in T$, such that $d_{\circ,L}(a,b)=\ell_{\circ,T}(t\circ \lambda^{-1}_t(b)\circ b^\dagger)$ and $d_{\circ,L}(b,a)=\ell_{\circ,T}(b\circ \lambda_t^{-1}(b)^{\dagger}\circ t^\dagger)$. Consequently, if $\ell_{\circ,T}(b\circ \lambda_t^{-1}(b)^{\dagger}\circ t^\dagger)=k$, then the graph contains a directed $(k+1)$-cycle consisting of $k$ multiplicative edges and one additive edge.  
\end{enumerate}
\end{theorem}
\begin{proof}
(1) Since $d_{\circ,L}(a,b) = 1$, there exists an element $t \in T$ such that $a = t \circ b$. Applying $x \circ y = x \cdot \lambda_x(y)$, we obtain $a = t \cdot \lambda_t(b)$, which implies $a \cdot b^{-1} = t \cdot \lambda_t(b) \cdot b^{-1}$. Suppose $d_{\cdot,L}(a,b) = m$. This means that there are intermediate vertices $u_1, ..., u_{m-1}$ and generators $s_1, s_2, ..., s_m \in T$ such that $s_1 = a \cdot u_1^{-1}$, $s_m = u_{m-1} \cdot b^{-1}$, and $s_j = u_{j-1} \cdot u_j^{-1}$ for $2 \le j \le m-1$. Through backward substitutions, we have
\[
a = (s_1 \cdot s_2 \cdot ... \cdot s_m) \cdot b \iff t \cdot \lambda_t(b) \cdot b^{-1} = a \cdot b^{-1} = s_1 \cdot s_2 \cdot ... \cdot s_m,
\]
which corresponds to the shortest factorization of length $m$ under the additive operation $\cdot$. Thus, we conclude that $d_{\cdot,L}(a,b) = \ell_{\cdot,T}(t \cdot \lambda_t(b) \cdot b^{-1})$.

It is obvious that if $\langle T \rangle_\cdot = (A,\cdot)$ and $A$ is finite then, there exists $r_1,r_2,...,r_k\in T$ such that $b\cdot a^{-1}=r_1\cdot r_2 \cdot ... \cdot r_k$. This implies $d_{\cdot,L}(b,a) = \ell_{\cdot,T}(b \cdot \lambda_t(b)^{-1} \cdot t^{-1})$. Finally, concatenating the the shortest additive path from $b$ to $a$ with multiplicative edge $(a,b) \in E_{\circ,L}$ forms a directed $(k+1)$-cycle in $\mathrm{Cay}_L((A, \cdot, \circ), T)$, consisting of $k$ additive edges and one multiplicative edge.

(2) By using the relation $x \cdot y = x \circ \lambda_x^{-1}(y)$, the proof follows immediately.
\end{proof}
\vspace{0.5 cc}
\begin{theorem} \label{theo:mixed-edge-additive/multiplicative-path-ab-right-cayley}
Let $\mathrm{Cay}_R((A, \cdot, \circ), T)$ be a right Cayley graph of a finite skew brace $(A, \cdot, \circ)$ with connection set $T \subseteq A$. For any $a, b \in A$, the following hold: 
\begin{enumerate}
    \item If $\langle T\rangle_\cdot=(A,\cdot)$ and $d_{\circ,R}(a,b)=1$, then $a = b \circ t$ for some $t \in T$, such that $d_{\cdot,R}(a,b) = \ell_{\cdot,T}(\lambda_b(t))$ and $d_{\cdot,R}(b,a)= \ell_{\cdot,T}(\lambda_b(t)^{-1})$. Consequently, if $\ell_{\cdot,T}(\lambda_b(t)^{-1}) = k$, then the graph contains a directed $(k+1)$-cycle consisting of $k$ additive edges and one multiplicative edge.
    \item If $\langle T\rangle_\circ=(A,\circ)$ and $d_{\cdot,R}(a,b)=1$, then $a = b\cdot t$ for some $t \in T$, such that $d_{\circ,R}(a,b)=\ell_{\circ,T}(\lambda^{-1}_b(t))$ and $d_{\circ,R}(b,a)=\ell_{\circ,T}(\lambda_b^{-1}(t)^{\dagger})$. Consequently, if $\ell_{\circ,T}(\lambda_b^{-1}(t)^{\dagger})=k$, then the graph contains a directed $(k+1)$-cycle consisting of $k$ multiplicative edges and one additive edge.  
\end{enumerate}
\end{theorem}
\begin{proof}
The proof is similar to the proof of Theorem~\ref{theo:mixed-edge-additive/multiplicative-path-ab-left-cayley}.
\end{proof}

\vspace{.5 cc}

\subsection{Ideals and Cayley Graphs of Skew Braces }\mbox{}\par
We characterize ideals of a finite skew brace using distances in the corresponding Cayley graphs. We have the following result.
\vspace{0.5 cc}
\begin{theorem}\label{theorem_ideal}
Let $(A, \cdot, \circ)$ be a finite skew brace and the connection $T$ be a subgroup of $(A, \cdot)$. The following are equivalent:
\begin{enumerate}
    \item $T$ is an ideal of $A$.
    \item $d_{\cdot, L}(a,a\circ t) = d_{\cdot, R}(a,a\circ t) = d_{\circ, L}(a,a\circ t) = d_{\circ,R}(a,a\circ t) = 1$, for all $a\in A$ and $t\in T$.
    \item $E_{\cdot, L} = E_{\cdot, R} = E_{\circ, R} = E_{\circ, L}$.
\end{enumerate}
\end{theorem}

\begin{proof}
We will show that (1) $\Longleftrightarrow$ (2) and (2) $\Longleftrightarrow$ (3).\vspace*{0.1cm}

\noindent $\underline{\mathrm{(1)}\Rightarrow\mathrm{(2)}}$. Suppose $T$ is an ideal of $A$, and let $a\in A$ and $t\in T$ with $t\neq 1$. We show that $a\circ t$ is at distance $1$ from $a$ in all four Cayley graphs.

Since $T$ is an ideal, $\lambda_a(t)\in T$, and $T$ is a normal subgroup of both $(A,\cdot)$ and $(A,\circ)$. Therefore $  \lambda_a(t)^{-1}, a\,\lambda_a(t)^{-1}a^{-1}, t^{\dagger}, a\circ t^{\dagger}\circ a^{\dagger}$ are non-identity element in $T$. By Theorem \ref{theo:distance-cayley}, we conclude that
\begin{align*}
  d_{\cdot,R}(a,a\circ t) &= \ell_{\cdot,T}(\lambda_a(t)^{-1}) = 1, &
  d_{\cdot,L}(a,a\circ t) &= \ell_{\cdot,T}(a\,\lambda_a(t)^{-1}a^{-1}) = 1,\\
  d_{\circ,R}(a,a\circ t) &= \ell_{\circ,T}(t^{\dagger}) = 1, &
  d_{\circ,L}(a,a\circ t) &= \ell_{\circ,T}(a\circ t^{\dagger}\circ a^{\dagger}) = 1.
\end{align*}

\noindent $\underline{\mathrm{(2)}\Rightarrow\mathrm{(1)}}$. Assume (2), and let $a\in A$ and $t\in T$. We may assume $t\neq 1$, since every condition below is trivial for $t=1$. Recall that $\ell_{\cdot,T}(x)=1$ implies $x\in T$, and similarly for $\ell_{\circ,T}$.

From $d_{\cdot,R}(a,a\circ t)=\ell_{\cdot,T}\big(\lambda_a(t)^{-1}\big)=1$ we get $\lambda_a(t)^{-1}\in T$, hence $\lambda_a(t)\in T$, as $(T,\cdot)$ is a subgroup of $(A,\cdot)$. Thus $\lambda_a(T)\subseteq T$, and since $\lambda_a$ is injective and $T$ is finite, $\lambda_a(T)=T$. Consequently, $(T,\circ)$ is a subgroup of $(A,\circ)$ since for any $s,t\in T$, we have $s\circ t = s\,\lambda_s(t)\in T$, and $t^{\dagger}$ is in $T$ since $d_{\circ, R}(a,a\circ t) = \ell_{\circ, T}(t^{\dag})=1$.

From $d_{\cdot,L}(a,a\circ t)=\ell_{\cdot,T}\big(a\,\lambda_a(t)^{-1}a^{-1}\big)=1$ we get $a\,\lambda_a(t)^{-1}a^{-1}\in T$. As $t$ runs over $T$, the element $\lambda_a(t)^{-1}$ runs over all of $T$, because $\lambda_a(T)=T$. Hence $aTa^{-1}\subseteq T$, so $(T,\cdot)$ is normal in $(A,\cdot)$.

Finally, $d_{\circ,L}(a,a\circ t)=\ell_{\circ,T}\big(a\circ t^{\dagger}\circ a^{\dagger}\big)=1$ gives $a\circ t^{\dagger}\circ a^{\dagger}\in T$ for all $t\in T$. Hence $(T,\circ)$ is normal in $(A,\circ)$.

Therefore $T$ is an ideal of $A$.\vspace*{0.2cm}

\noindent $\underline{\mathrm{(1)}\Rightarrow\mathrm{(3)}}$. Suppose $T$ is an ideal of $A$. By definition of the Cayley graphs, for $u,v\in A$ we have
\[
  (u,v)\in E_{\cdot,R} \iff u\in vT, \qquad
  (u,v)\in E_{\cdot,L} \iff u\in Tv,
\]
\[
  (u,v)\in E_{\circ,R} \iff u\in v\circ T, \qquad
  (u,v)\in E_{\circ,L} \iff u\in T\circ v.
\]
It therefore suffices to show that $vT = Tv = v\circ T = T\circ v$ for every $v\in A$. The equalities $vT=Tv$ and $v\circ T = T\circ v$ hold because $T$ is normal in $(A,\cdot)$ and in $(A,\circ)$, respectively. Moreover, $\lambda_v(T)\subseteq T$ and $\lambda_{v^{\dagger}}(T)\subseteq T$ give $\lambda_v(T)=T$, so
\[
  v\circ T = v\,\lambda_v(T) = vT.
\]
Hence $E_{\cdot,R}=E_{\cdot,L}=E_{\circ,R}=E_{\circ,L}$. 
\vspace*{0.2cm}

\noindent $\underline{\mathrm{(3)}\Rightarrow\mathrm{(1)}}$. Suppose $E_{\cdot, L} = E_{\cdot, R} = E_{\circ, R} = E_{\circ, L}$. We show that $T$ is an ideal by showing $(T,\circ)$ is a normal subgroup of $(A,\circ)$, $Tv=vT$, and $\lambda_v(T)\subseteq T$ for any $v\in A$.

First, we show that $(T,\circ)$ is a normal subgroup of $(A,\circ)$. Let $t\in T$, and $v\in A$, and write $u =  v\circ t$ in $A$. Therefore $v^{\dag}\circ u = t\in A$, and thus $(u,v) \in E_{\circ, R}$. By our assumption, $E_{\circ, R} = E_{\circ, L}$, therefore there is $t'\in T$ such that $u=t'\circ v  = v\circ t$. Hence, $ v\circ t \circ v^{\dag}=t'\in T$, and thus $(T,\circ)$ is a normal subgroup of $A$.

Next, we show that $Tv = vT$ for any $v\in A$. For any $t\in T$, write $u = tv\in Tv\subseteq A$, therefore $uv^{-1} =t\in T$ and thus $(u,v)\in E_{\cdot, L}$. By our assumption, $E_{\cdot, L} = E_{\cdot, R}$, therefore there is $t'\in T$ such that $u = tv = vt'$, and hence $Tv \subseteq vT$. By similar argument, one can also show that $vT \subseteq Tv$. We conclude that $Tv = vT$

Finally, we show that $\lambda_v(T)\subseteq T$ for any $v\in A$. For any $t\in T\subseteq A$, we have $v \lambda_v(t) = v\circ t$. Write $u=v\circ t$, then it means $(u,v)\in E_{\circ, R}$. By our assumption, $E_{\circ, R} = E_{\cdot, R}$, therefore, there is $t' \in T$ such that $v\circ t = v\lambda_v(t) = vt'$. Thus, for any $t\in T$, $\lambda_v(t) = t'$ for some $t'\in T$ and we conclude that $\lambda_v(T)\subseteq T$.\vspace*{0.2cm}

This completes our proof.
\end{proof}
\vspace{0.5 cc}
\begin{example} \label{ex:skew_brace_s3}
Let $(S_3, \cdot,\circ)$ be a skew brace on $S_3$ as explained in Example~\ref{ex:skew_brace_s3} and $T = H = \{\mathrm{1}, (123), (132)\}$ be a connection set. One can verify that $T$ is an ideal of skew brace $S_3$ (by Definition~\ref{def:ideal_skewbrace}) and the equivalences established in Theorem~\ref{theorem_ideal}. 
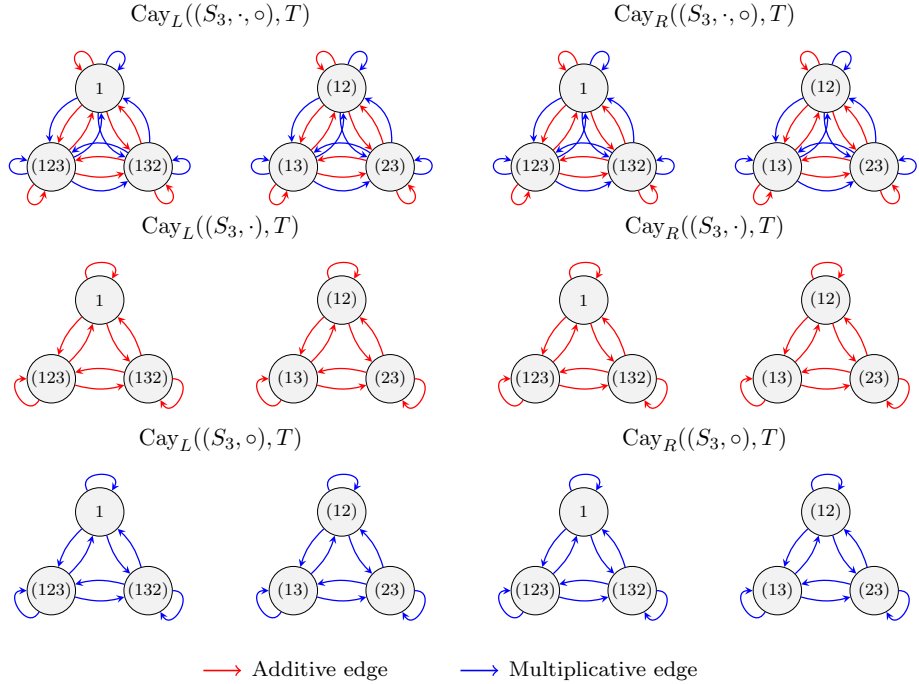
\begin{figure}[H]
    \centering
    \resizebox{0.9\textwidth}{!}{
    \newcommand{\drawbichromaticclique}[3]{
        \draw[->, semithick, red, bend right=15] (#1) to (#2); \draw[->, semithick, red, bend right=15] (#2) to (#1);
        \draw[->, semithick, red, bend right=15] (#2) to (#3); \draw[->, semithick, red, bend right=15] (#3) to (#2);
        \draw[->, semithick, red, bend right=15] (#3) to (#1); \draw[->, semithick, red, bend right=15] (#1) to (#3);
        
        \draw[->, semithick, blue, bend right=35] (#1) to (#2); \draw[->, semithick, blue, bend right=35] (#2) to (#1);
        \draw[->, semithick, blue, bend right=35] (#2) to (#3); \draw[->, semithick, blue, bend right=35] (#3) to (#2);
        \draw[->, semithick, blue, bend right=35] (#3) to (#1); \draw[->, semithick, blue, bend right=35] (#1) to (#3);
        
        \draw[->, semithick, red, out=105, in=135, loop, distance=4mm] (#1) to (#1);
        \draw[->, semithick, blue, out=75, in=45, loop, distance=4mm] (#1) to (#1);
        
        \draw[->, semithick, red, out=225, in=255, loop, distance=4mm] (#2) to (#2);
        \draw[->, semithick, blue, out=195, in=165, loop, distance=4mm] (#2) to (#2);
        
        \draw[->, semithick, red, out=285, in=315, loop, distance=4mm] (#3) to (#3);
        \draw[->, semithick, blue, out=345, in=15, loop, distance=4mm] (#3) to (#3);
    }
    
    \newcommand{\drawmonochromaticclique}[4]{
        \draw[->, semithick, #4, bend right=15] (#1) to (#2); \draw[->, semithick, #4, bend right=15] (#2) to (#1);
        \draw[->, semithick, #4, bend right=15] (#2) to (#3); \draw[->, semithick, #4, bend right=15] (#3) to (#2);
        \draw[->, semithick, #4, bend right=15] (#3) to (#1); \draw[->, semithick, #4, bend right=15] (#1) to (#3);
        
        \draw[->, semithick, #4, out=120, in=60, loop, distance=4mm] (#1) to (#1);
        \draw[->, semithick, #4, out=240, in=180, loop, distance=4mm] (#2) to (#2);
        \draw[->, semithick, #4, out=360, in=300, loop, distance=4mm] (#3) to (#3);
    }

    \begin{tikzpicture}[
        every node/.style={circle, draw=black, fill=gray!10, font=\scriptsize, minimum size=8mm, inner sep=0.5pt},
        >=stealth
    ]
    
    \begin{scope}[xshift=-4cm]
        \node[draw=none, fill=none, rectangle, font=\normalsize] at (0, 2) {$\mathrm{Cay}_L((S_3, \cdot, \circ), T)$};
        
        \node (1a) at (-2.0, 0.8) {$1$};
        \node (2a) at (-2.8, -0.5) {$(123)$};
        \node (3a) at (-1.2, -0.5) {$(132)$};
        
        \node (4a) at (2.0, 0.8) {$(12)$};
        \node (5a) at (1.2, -0.5) {$(13)$};
        \node (6a) at (2.8, -0.5) {$(23)$};
        
        \drawbichromaticclique{1a}{2a}{3a}
        \drawbichromaticclique{4a}{5a}{6a}
    \end{scope}

    \begin{scope}[xshift=4cm]
        \node[draw=none, fill=none, rectangle, font=\normalsize] at (0, 2) {$\mathrm{Cay}_R((S_3, \cdot, \circ), T)$};
        
        \node (1b) at (-2.0, 0.8) {$1$};
        \node (2b) at (-2.8, -0.5) {$(123)$};
        \node (3b) at (-1.2, -0.5) {$(132)$};
        
        \node (4b) at (2.0, 0.8) {$(12)$};
        \node (5b) at (1.2, -0.5) {$(13)$};
        \node (6b) at (2.8, -0.5) {$(23)$};
        
        \drawbichromaticclique{1b}{2b}{3b}
        \drawbichromaticclique{4b}{5b}{6b}
    \end{scope}

    \begin{scope}[xshift=-4cm, yshift=-3.5cm]
        \node[draw=none, fill=none, rectangle, font=\normalsize] at (0, 2) {$\mathrm{Cay}_L((S_3, \cdot), T)$};
        
        \node (1c) at (-2.0, 0.8) {$1$};
        \node (2c) at (-2.8, -0.5) {$(123)$};
        \node (3c) at (-1.2, -0.5) {$(132)$};
        
        \node (4c) at (2.0, 0.8) {$(12)$};
        \node (5c) at (1.2, -0.5) {$(13)$};
        \node (6c) at (2.8, -0.5) {$(23)$};
        
        \drawmonochromaticclique{1c}{2c}{3c}{red}
        \drawmonochromaticclique{4c}{5c}{6c}{red}
    \end{scope}

    \begin{scope}[xshift=4cm, yshift=-3.5cm]
        \node[draw=none, fill=none, rectangle, font=\normalsize] at (0, 2) {$\mathrm{Cay}_R((S_3, \cdot), T)$};
        
        \node (1d) at (-2.0, 0.8) {$1$};
        \node (2d) at (-2.8, -0.5) {$(123)$};
        \node (3d) at (-1.2, -0.5) {$(132)$};
        
        \node (4d) at (2.0, 0.8) {$(12)$};
        \node (5d) at (1.2, -0.5) {$(13)$};
        \node (6d) at (2.8, -0.5) {$(23)$};
        
        \drawmonochromaticclique{1d}{2d}{3d}{red}
        \drawmonochromaticclique{4d}{5d}{6d}{red}
    \end{scope}

    \begin{scope}[xshift=-4cm, yshift=-7cm]
        \node[draw=none, fill=none, rectangle, font=\normalsize] at (0, 2) {$\mathrm{Cay}_L((S_3, \circ), T)$};
        
        \node (1e) at (-2.0, 0.8) {$1$};
        \node (2e) at (-2.8, -0.5) {$(123)$};
        \node (3e) at (-1.2, -0.5) {$(132)$};
        
        \node (4e) at (2.0, 0.8) {$(12)$};
        \node (5e) at (1.2, -0.5) {$(13)$};
        \node (6e) at (2.8, -0.5) {$(23)$};
        
        \drawmonochromaticclique{1e}{2e}{3e}{blue}
        \drawmonochromaticclique{4e}{5e}{6e}{blue}
    \end{scope}

    \begin{scope}[xshift=4cm, yshift=-7cm]
        \node[draw=none, fill=none, rectangle, font=\normalsize] at (0, 2) {$\mathrm{Cay}_R((S_3, \circ), T)$};
        
        \node (1f) at (-2.0, 0.8) {$1$};
        \node (2f) at (-2.8, -0.5) {$(123)$};
        \node (3f) at (-1.2, -0.5) {$(132)$};
        
        \node (4f) at (2.0, 0.8) {$(12)$};
        \node (5f) at (1.2, -0.5) {$(13)$};
        \node (6f) at (2.8, -0.5) {$(23)$};
        
        \drawmonochromaticclique{1f}{2f}{3f}{blue}
        \drawmonochromaticclique{4f}{5f}{6f}{blue}
    \end{scope}
    
    \end{tikzpicture}
    } 

    \vspace{0.2cm}
    \begin{tikzpicture}
        \draw[->, semithick, red] (0,0) -- (0.5,0) node[right, text=black, font=\scriptsize, draw=none, fill=none] {Additive edge};
        \draw[->, semithick, blue] (3.4,0) -- (3.9,0) node[right, text=black, font=\scriptsize, draw=none, fill=none] {Multiplicative edge};
    \end{tikzpicture}

    \caption{The Cayley graphs of the skew brace $S_3$ where the connection set $T=\{1, (123), (132)\}$ is an ideal.}
    \label{fig:cayley-s3-ideal}
\end{figure}
\end{example}

\vspace{1.5 cc}
\section{CONCLUDING REMARKS}

In this paper, we have introduced Cayley graphs of finite skew braces as bichromatic directed multigraphs. We established that distances in these graphs correspond to minimal factorization lengths under the $\lambda$-map. This approach yields a characterization of the ideals of a skew brace in terms of their corresponding Cayley graphs. From an algebraic graph-theoretic perspective, this offers an alternative method to identify these ideals, substantially reducing the computational complexity required for large skew braces.

The Cayley graphs of finite skew braces established here suggest several directions for future research as follows: 
\begin{enumerate}
    \item Generalizing our results related to paths and distances by considering both operations simultaneously, allowing the characterization of the graph's diameter.
    \item Establishing upper bounds for this diameter based on the properties of the $\lambda$-map.
    \item Investigating whether the set-theoretic solutions to the YBE can be characterized via the structures of Cayley graphs of skew braces.
    \item Considering $d_{mix}$ as a metric on skew braces, thereby equipping skew braces with a metric structure, and relating this metric framework to word metrics on free skew braces. Additionally, analyzing word metrics and minimal factorization lengths within these graphs holds potential for practical applications, including cryptographic and combinatorial optimization problems involving path routing in algebraic networks.
    \item Comparing the regularity of the underlying simple graph of Cayley graphs of finite skew braces with the regularity of Cayley graphs of groups.
    \item Following \cite{guarnieri2017skew}, for any skew brace $(A,\cdot, \circ)$, there is a corresponding semidirect product group $\Lambda = (A, \cdot) \rtimes_\lambda (A, \circ)$. Therefore, it is natural to ask how the Cayley graph of group $\Lambda$ relates to the Cayley graph of skew brace $A$.  
\end{enumerate} 
\vspace{1.5cc}

\section*{\centering DECLARATIONS}
\addcontentsline{toc}{section}{Declarations} 

\noindent The authors declare that they have no conflict of interest.

\vspace{1.5cc}

\section*{\centering ACKNOWLEDGMENTS}
\addcontentsline{toc}{section}{Acknowledgments} 

\noindent The authors acknowledge the use of AI to assist with English translation, language refinement, and figure coding. All generated content has been thoroughly reviewed and verified by the authors, who assume full responsibility for the final content.

\vspace{1cm}

\renewcommand{\refname}{REFERENCES} 
\bibliographystyle{ieeetr}
\bibliography{references}

\end{document}